\documentclass[11pt,a4paper]{article}

\usepackage{graphicx,latexsym,euscript,makeidx,color,bm}
\usepackage{amsmath,amsfonts,amssymb,amsthm,thmtools,mathrsfs,enumerate}
\usepackage[colorlinks,linkcolor=black,anchorcolor=black,citecolor=black]{hyperref}
\usepackage{geometry}
\def\dbE{\mathbb{E}}     
\def\dbF{\mathbb{F}}   \def\cF{{\cal F}}  
     
   \def\cH{{\cal H}}

 \def\sN{\mathscr{N}}    
     
\def\dbP{\mathbb{P}}     
     
\def\dbR{\mathbb{R}} \def\sR{\mathscr{R}}    
\def\dbS{\mathbb{S}}

      \def\lt{\left}       \def\hb{\hbox}
\def\ms{\medskip}        \def\rt{\right}      \def\ae{\hbox{\rm a.e.}}
        \def\lan{\langle}    \def\as{\hbox{\rm a.s.}}
   \def\ran{\rangle}    \def\tr{\hbox{\rm tr$\,$}}
\def\ts{\textstyle}         \def\rank{\hbox{\rm rank\,}}
\def\no{\noindent}

\def\rf{\eqref}            
         
\def\deq{\triangleq}     \def\({\Big (}       
\def\les{\leqslant}      \def\){\Big )}       %\def\tc{\textcolor{red}}
\def\ges{\geqslant}      \def\[{\Big[}        %\def\tb{\textcolor{blue}}
          \def\]{\Big]}        \def\ba{\begin{array}}
      \def\q{\quad}        \def\ea{\end{array}}
\def\h{\widehat}         \def\qq{\qquad}      \def\1n{\negthinspace}
           \def\2n{\1n\1n}      \def\bel{\begin{equation}\label}
         \def\3n{\1n\2n}      \def\ee{\end{equation}}
\def\hp{\hphantom}       \def\nn{\nonumber}   \def\cl{\overline}

        \def\G{\Gamma}      \def\Om{\Omega}   \def\om{\omega}
\def\b{\beta}            \def\d{\delta}   \def\F{\Phi}     
           \def\th{\theta}  \def\Si{\Sigma}  
\def\e{\varepsilon}     \def\l{\lambda}        
         \def\f{\varphi}  \def\i{\infty}   

\newtheoremstyle{indented}{}{}{\it}{\parindent}{\bfseries}{.}{.5em}{}
\theoremstyle{indented}

\newtheorem{theorem}{Theorem}[section]
\newtheorem{definition}[theorem]{Definition}
\newtheorem{proposition}[theorem]{Proposition}
\newtheorem{corollary}[theorem]{Corollary}
\newtheorem{lemma}[theorem]{Lemma}
\newtheorem{remark}[theorem]{Remark}

\newtheorem{taggedthmx}{}
\newenvironment{taggedthm}[1]
 {\renewcommand\thetaggedthmx{#1}\taggedthmx}
 {\endtaggedthmx}

\makeatletter
   
   \@addtoreset{equation}{section}
   \newcommand{\setword}[2]{%
   \phantomsection
   #1\def\@currentlabel{\unexpanded{#1}}\label{#2}%
   }
\makeatother
\allowdisplaybreaks[4]
\begin{document}

\title{\bf Optimal Control for Controllable Stochastic Linear Systems}
\author{
 Xiuchun Bi\thanks{Department of Statistics and Finance, University of Science and Technology of China,
                   Hefei, Anhui, 230026, China (xcbi@ustc.edu.cn).} \and
Jingrui Sun\thanks{Corresponding author. Department of Mathematics, Southern University of Science and Technology,
                   Shenzhen, Guangdong, 518055, China (sunjr@sustech.edu.cn).} \and
  Jie Xiong\thanks{Department of Mathematics, Southern University of Science and Technology,
                   Shenzhen, Guangdong, 518055, China (xiongj@sustech.edu.cn).}
}

%\date{}
\maketitle

\no{\bf Abstract.}
This paper is concerned with a constrained stochastic linear-quadratic optimal control problem,
in which the terminal state is fixed and the initial state is constrained to lie in a stochastic linear manifold.
The controllability of stochastic linear systems is studied.
Then the optimal control is explicitly obtained by considering a parameterized unconstrained backward LQ problem
and an optimal parameter selection problem.
A notable feature of our results is that, instead of solving an equation involving derivatives with respect to
the parameter, the optimal parameter is characterized by an {\it algebraic} equation.

\ms

\no{\bf Key words.}
linear-quadratic, optimal control, controllability, controllability Gramian, Lagrange multiplier,
optimal parameter, Riccati equation.

\ms

\no{\bf AMS subject classifications.}
49N10, 93E20, 93E24, 93B05.

\section{Introduction}\label{Sec:Intro}

Let $(\Om,\cF,\dbP)$ be a complete probability space on which a standard one-dimensional Brownian motion
$W=\{W(t);0\les t<\i\}$ is defined, and let $\dbF=\{\cF_t\}_{t\ges0}$ be the natural filtration of $W$
augmented by all the $\dbP$-null sets in $\cF$.
Consider the following controlled linear stochastic differential equation (SDE, for short) on a finite
horizon $[t,T]$:
\bel{state} dx(s) = [A(s)x(s)+B(s)u(s)]ds + [C(s)x(s)+D(s)u(s)]dW(s), \q s\in[t,T], \ee
where $A,C:[0,T]\to\dbR^{n\times n}$ and $B,D:[0,T]\to\dbR^{n\times m}$, called the
{\it coefficients} of the {\it state equation} \rf{state}, are given deterministic functions.
The solution $x=\{x(s);t\les s\les T\}$ of \rf{state}, which takes values in $\dbR^n$, is called a
{\it state process}, and the process $u=\{u(s);t\les s\les T\}$, which takes values in $\dbR^m$ and is
$\dbF$-progressively measurable, is called a {\it control}.
For a given initial condition $x(t)=\xi$, the state process $x$ is uniquely determined by the control $u$,
and is often denoted by $x^{t,\xi,u}$ when it is necessary to underline the dependence on the {\it initial pair}
$(t,\xi)$ and the control $u$.
In this paper we shall assume that the coefficients of the state equation \rf{state} satisfy the following condition:
\begin{itemize}
\item[{\setword{\bf(A1)}{(A1)}}] $A,C:[0,T]\to\dbR^{n\times n}$ and $B,D:[0,T]\to\dbR^{n\times m}$
are bounded, Lebesgue measurable functions.
\end{itemize}
According to the standard result for SDEs (see, for example, \cite[Chapter 1, Theorem 6.3]{Yong-Zhou 1999}),
such a condition ensures that a unique $p$th power integrable solution exists for the SDE \rf{state} whenever
the {\it initial state} $x(t)=\xi$ and the control $u$ are $p$th power integrable.
We are interested in the case $p=2$, in which the spaces of initial states, {\it admissible controls} and
state processes are
\begin{align*}
L^2_{\cF_t}(\Om;\dbR^n)
  &\ts=\Big\{\xi:\Om\to\dbR^n\bigm|\xi~\hbox{is $\cF_t$-measurable with}~\dbE|\xi|^2<\i\Big\},\\
L_\dbF^2(t,T;\dbR^m)
  &\ts=\Big\{u:[t,T]\times\Om\to\dbR^m\bigm|u~\hbox{is $\dbF$-progressively measurable} \\
  &\ts\hp{=\Big\{\ } \hbox{with}~\dbE\int_t^T|u(s)|^2ds<\i\Big\}, ~\hbox{and}\\
L_\dbF^2(\Om;C([t,T];\dbR^n))
   &\ts=\Big\{x:[t,T]\times\Om\to\dbR^n\bigm|x~\hbox{is $\dbF$-adapted and continuous} \\
   &\ts\hp{=\Big\{\ } \hbox{with}~\dbE\big[\sup_{t\les s\les T}|x(s)|^2\big]<\i\Big\},
\end{align*}
respectively.

\medskip

Let $F\in\dbR^{k\times n}$ $(k\les n)$ be a matrix, and let $b\in L^2_{\cF_t}(\Om;\dbR^k)$ be a random variable.
We denote by $\cH(F,b)$ the {\it stochastic linear manifold}
$$\{\xi\in L^2_{\cF_t}(\Om;\dbR^n): F\xi=b\}.$$
The problems of interest here are those for which the control $u$ is required to drive the system \rf{state}
to a particular state at the end of the interval $[t,T]$ from a given stochastic linear manifold $\cH(F,b)$ and the cost
functional is of the quadratic form
%%
%\begin{equation}\label{cost}
%J(t,u) = \dbE\left\{\lan Gx(t),x(t)\ran
%         + \int_t^T\llan\1n\begin{pmatrix}Q(s) & \1nS(s)^\top \\
%                                          S(s) & \1nR(s)\end{pmatrix}\1n
%                           \begin{pmatrix}x(s) \\ u(s)\end{pmatrix}\1n,
%                           \begin{pmatrix}x(s) \\ u(s)\end{pmatrix}\1n\rran ds\right\},
%\end{equation}
%%
%
\begin{equation}\label{cost}
J(t,u) = \dbE\left\{\lan Gx(t),x(t)\ran
         + \int_t^T\lan Q(s)x(s),x(s)\ran+ \lan R(s)u(s),u(s)\ran ds\right\},
\end{equation}
where the {\it weighting matrices} $G$, $Q$, and $R$ are assumed to satisfy the following condition:
\begin{itemize}
\item[{\setword{\bf(A2)}{(A2)}}] $G\in\dbR^{n\times n}$ is a symmetric matrix;
$Q:[0,T]\to\dbR^{n\times n}$ and $R:[0,T]\to\dbR^{m\times m}$ are bounded, symmetric functions.
Moreover, for some real number $\d>0$,
$$G\ges 0, \q Q(s)\ges0, \q R(s)\ges \d I_m, \q\ae~s\in[0,T].$$
\end{itemize}
%
%%
%\begin{itemize}
%\item[{\setword{(A2)}{(A2)}}] \tc{The random matrix $G:\Om\to\dbR^{n\times n}$}
%is bounded, symmetric, and $\cF_t$-measurable;
%The processes $Q:[0,T]\times\Om\to\dbR^{n\times n}$, $S:[0,T]\times\Om\to\dbR^{m\times n}$, and $R:[0,T]\times\Om\to\dbR^{m\times m}$
%are bounded and $\dbF$-progressively measurable, with $Q$ and $R$ being symmetric.
%Moreover, for some real number $\d>0$,
%%
%$$G\ges 0, \q R(s)\ges \d I_m, \q Q(s)-S(s)^\top R(s)^{-1}S(s)\ges0, \q\ae~s\in[0,T],~\as$$
%%
%\end{itemize}
%%
For a precise statement, we pose the following {\it constrained stochastic linear-quadratic (LQ, for short) optimal control problem}.

\begin{taggedthm}{Problem (CLQ)}\label{Prob}
For a given target $\eta\in L^2_{\cF_T}(\Om;\dbR^n)$, find a control
$u^*\in L_\dbF^2(t,T;\dbR^m)$ such that the cost functional $J(t,u)$ is minimized over
$L_\dbF^2(t,T;\dbR^m)$, subject to the following constraints on the initial and terminal states:
\bel{constraint} x(t)\in\cH(F,b), \q x(T)=\eta.\ee
\end{taggedthm}

A control $u^*\in L_\dbF^2(t,T;\dbR^m)$ that minimizes $J(t,u)$ subject to \rf{constraint} will be
called an {\it optimal control} with respect to the target $\eta$;
the corresponding state process will be called an {\it optimal state process}.
If an initial state $\xi\in\cH(F,b)$ is transferred to the target $\eta$ by an optimal control,
we call $\xi$ an {\it optimal initial state}.

%We shall be concerned with not only the question of finding an optimal control,
%but also how to find an optimal initial state.

\medskip

If the constraint \rf{constraint} is absent, but the initial state $x(t)=\xi$ is given,\autoref{Prob}
becomes a standard stochastic LQ optimal control problem.
Such kind of problems was initiated by Wonham \cite{Wonham 1968} and was later investigated by many
researchers; see, for example, Bismut \cite{Bismut 1978}, Bensoussan \cite{Bensoussan 1982},
Chen and Yong \cite{Chen-Yong 2001}, Ait Rami, Moore, and Zhou \cite{Ait Rami-Moore-Zhou 2002},
Tang \cite{Tang 2003}, Yu \cite{Yu 2013}, Sun, Li, and Yong \cite{Sun-Li-Yong 2016},
L\"{u}, Wang, and Zhang \cite{Lu-Wang-Zhang 2017}, Sun, Xiong, and Yong \cite{Sun-Xiong-Yong 2019},
Wang, Sun, and Yong \cite{Wang-Sun-Yong 2019}, and the references therein.
In contrast, much less progress has been made on the constrained LQ problem for stochastic systems.
This problem is particularly difficult in the stochastic setting since not only is one required to
decide whether a state of the stochastic system can be transferred to another state,
but in addition an optimal parameter must be evaluated.

\medskip

There were some attempts in attacking the constrained stochastic LQ optimal control problem
in the special case of norm optimal control; see, for instance, Gashi \cite{Gashi 2015},
Wang and Zhang \cite{Wang-Zhang 2015}, and Wang et al. \cite{Wang-Yang-Yong-Yu 2017}.
However, in these works the state process is required to start from a {\it particular point},
and the optimal control is either characterized {\it implicitly} in terms of {\it coupled}
forward-backward stochastic differential equations (FBSDEs, for short), which are difficult to solve,
or explicitly obtained but under a {\it strong} assumption that the stochastic system is
{\it exactly controllable} (which means a target can be reached from any initial state).

\medskip

This paper aims to provide a complete solution to\autoref{Prob}, a class of stochastic LQ
optimal control problems with fixed terminal states.
A distinctive feature of the problem under consideration is that the state process is allowed
to start from a stochastic linear manifold $\cH(F,b)$, instead of a fixed initial state.
Clearly, our problem contains the norm optimal control as a particular case.
Another feature is that the stochastic system is {\it not} assumed to be exactly controllable.
The initial states outside the stochastic linear manifold $\cH(F,b)$ are irrelevant to our problem,
so figuring out when the target can be reached from $\cH(F,b)$ will be enough to tackle\autoref{Prob}.

\medskip

The principal method adopted in the paper is combination of Lagrange multipliers and
unconstrained backward LQ problems.
By introducing a parameter $\l$, the Lagrange multiplier,\autoref{Prob} is reduced to a parameterized
unconstrained backward LQ problem, whose optimal control and value function $V_\l$ can be
constructed explicitly using the solutions to a Riccati equation and a decoupled FBSDE.
Then the optimal state process $x^*_\l$ of the derived backward LQ problem is proved to be
also optimal for\autoref{Prob} if the parameter $\l$ is such that $x^*_\l(t)\in\cH(F,b)$.
In order to find such a parameter, called an {\it optimal parameter}, a first idea is to solve
the equation ${d\over d\l}V_\l=0$.
However, this does not work well in our situation, due to the difficulty in computing the derivative of $V_\l$.
Our approach for finding the optimal parameter is based on a refinement (\autoref{prop:controllability})
of Liu and Peng's result \cite[Theorem 2]{Liu-Peng 2010}.
The key is to establish an equivalence relationship between the controllability of the original system
and a system involving $\Si$, the solution of a Riccati equation (\autoref{prop:x=hatx}).
By observing that the controllability Gramian of the new system is exactly $\Si(t)$ (\autoref{prop:Gramian-Si}),
we show that an optimal parameter can be obtained by solving an algebraic equation (\autoref{thm:main}).

\medskip

The rest of the paper is organized as follows.
In \autoref{Sec:Preliminary}, we collect some preliminary results.
\autoref{Sec:controllability} is devoted to the study of controllability of stochastic linear systems.
In \autoref{Sec:Lagrange}, using Lagrange multipliers, we reduce the problem to a parameterized
unconstrained backward LQ problem and an optimal parameter selection problem.
Finally, we discuss how to find an optimal parameter and present the complete solution to\autoref{Prob}
in \autoref{Sec:Main-result}.

\section{Preliminaries}\label{Sec:Preliminary}

Let $\dbR^{n\times m}$ be the Euclidean space consisting of $n\times m$ real matrices,
and let $\dbR^n=\dbR^{n\times1}$.
The inner product of $M,N\in\dbR^{n\times m}$, denoted by $\lan M,N\ran$, is given by
$\lan M,N\ran=\tr(M^\top N)$, where $M^\top$ is transpose of $M$ and $\tr(M^\top N)$
stands for the trace of $M^\top N$.
This inner product induces the Frobenius norm $|M|=\sqrt{\tr(M^\top M)}$.
Denote by $\dbS^n$ the space of all symmetric $n\times n$ real matrices,
and by $\dbS^n_+$ the space of all symmetric positive definite $n\times n$ real matrices.
For $\dbS^n$-valued functions $M$ and $N$, we write $M\ges N$ (respectively, $M>N$)
if $M-N$ is positive semidefinite (respectively, positive definite) almost everywhere.
The identity matrix of size $n$ is denoted by $I_n$.

\medskip

We now present some lemmas that are useful in the subsequent sections. Consider the linear BSDE
\bel{Y-Formula}\left\{\begin{aligned}
 dY(s) &= \big[A(s)Y(s)+C(s)Z(s)+f(s)\big]ds + Z(s)dW(s), \q s\in[0,T], \\
  Y(T) &= \eta.
\end{aligned}\right.\ee
The following result, coming from the idea of proving the well-posedness of linear BSDEs (see \cite[Chapter 7, Theorem 2.2]{Yong-Zhou 1999}),
provides a formula for the first component of the adapted solution $(Y,Z)$ to the BSDE \rf{Y-Formula}.

\begin{lemma}\label{lemma:Y-formula}
Let {\rm\ref{(A1)}} hold, and let $f\in L_\dbF^2(0,T;\dbR^n)$, $\eta\in L_{\cF_T}^2(\Om;\dbR^n)$.
Then the first component $Y$ of the adapted solution to \rf{Y-Formula} has the following representation:
$$ Y(t) = \dbE\lt[\G(t,T)\eta-\int_t^T\G(t,s)f(s)ds\bigg|\cF_t\rt], \q t\in[0,T], $$
where $\G(t,s)\deq\G(t)^{-1}\G(s)$ with $\G=\{\G(s);0\les s\les T\}$ being the solution to
$$\left\{\begin{aligned}
 d\G(s) &= -\G(s)A(s)ds -\G(s)C(s)dW(s), \q s\in[0,T], \\
  \G(0) &= I_n.
\end{aligned}\right.$$
\end{lemma}

\begin{proof}
Let $\th=\G(T)\eta-\int_0^T\G(s)f(s)ds$. By It\^{o}'s formula,
\begin{align*}
d\G Y  &=  -\G AYds - \G CYdW + \G(AY + CZ + f)ds + \G ZdW - \G CZds \\
       &=  \G fds +\G(Z-CY)dW,
\end{align*}
from which it follows that
\begin{align}\label{FY--17Feb18}
\G(t)Y(t)  &= \G(T)\eta - \int_t^T\G(s)f(s)ds - \int_t^T\G(s)\big[Z(s)-C(s)Y(s)\big]dW(s) \nn\\
           &= \th + \int_0^t\G(s)f(s)ds - \int_t^T\G(s)\big[Z(s)-C(s)Y(s)\big]dW(s).
\end{align}
Note that
$$\dbE\lt(\int_0^T\big|\G(s)[Z(s)-C(s)Y(s)]\big|^2ds\rt)^{1\over 2}<\i.$$
Hence, the process
$$M(t)\equiv\int_0^t\G(s)\big[Z(s)-C(s)Y(s)\big]dW(s),\q 0\les t\les T$$
is a martingale, and by taking conditional expectations with respect to $\cF_t$
on both sides of \rf{FY--17Feb18}, we obtain
$$\G(t)Y(t) = \dbE[\th|\cF_t] + \int_0^t\G(s)f(s)ds = \dbE\lt[\G(T)\eta-\int_t^T\G(s)f(s)ds\bigg|\cF_t\rt],$$
from which the desired result follows.
\end{proof}

We conclude this section with a simple but useful algebraic lemma.

\begin{lemma}\label{lmm:range=range}
Let $A\in\dbR^{m\times n}$ and $B\in\cl{\dbS^n_+}$. Then $ABA^\top$ and $AB$ have the same range space.
\end{lemma}

\begin{proof}
For a matrix $M$, let $\sR(M)$ and $\sN(M)$ denote the range and kernel of $M$, respectively.
Since $\sR(M)^\bot=\sN(M^\top)$ for any matrix $M$, it is suffice to prove
$$\sN(ABA^\top)=\sN(BA^\top).$$
Clearly, $\sN(BA^\top)\subseteq\sN(ABA^\top)$.
For the reverse inclusion, let $C\in\dbR^{n\times n}$ be such that $B=CC^\top$.
If $x\in\dbR^m$ is such that $ABA^\top x=0$, then
$$\big|C^\top A^\top x\big|^2=x^\top ACC^\top A^\top x=x^\top ABA^\top x=0.$$
Thus, $C^\top A^\top x=0$  and  hence $BA^\top x=CC^\top A^\top x=0$.
This shows that $\sN(ABA^\top)\subseteq\sN(BA^\top)$.
\end{proof}

%\begin{lemma}
%Let $A\in\dbR^{m\times n}$ with $\rank(A)=m$ and let $B\in\cl{\dbS^n_+}$. Then
%$\sR(A)=\sR(AB)$ if and only if $ABA^\top$ is invertible.
%\end{lemma}
%
%\begin{proof}
%By Lemma \ref{lmm:range=range}, $\sR(A)=\sR(AB)$ if and only if $\sR(A)=\sR(ABA^\top)$.
%Since $\rank(A)=m$, the range of $A$ is exactly $\dbR^m$. Thus $\sR(A)=\sR(ABA^\top)$
%if and only if $\sR(ABA^\top)=\dbR^m$, or equivalently, $ABA^\top$ is invertible.
%\end{proof}

\section{Controllability of linear stochastic systems}\label{Sec:controllability}

Consider the controlled linear stochastic differential system
\begin{equation}\label{[ACBD]}
dx(s)=[A(s)x(s)+B(s)u(s)]ds + [C(s)x(s)+D(s)u(s)]dW(s).
\end{equation}
Let $(t_0,x_0)\in [0,T)\times L^2_{\cF_{t_0}}(\Om;\dbR^n)$ be an initial pair,
and let $t_1\in(t_0,T]$ be the terminal time.
We know by the standard result for SDEs (\cite[Chapter 1, Theorem 6.3]{Yong-Zhou 1999})
that a solution $x^{t_0,x_0,u}\in L_\dbF^2(\Om;C([t_0,t_1];\dbR^n))$ uniquely exists
for any control $u\in L_\dbF^2(t_0,t_1;\dbR^m)$.
We are now concerned with the question of finding a control such that a given target
(terminal state) is reached on the terminal time.

\begin{definition}\rm
We say that a control $u\in L_\dbF^2(t_0,t_1;\dbR^m)$ {\it transfers the state of the system \rf{[ACBD]}
from $x_0\in L^2_{\cF_{t_0}}(\Om;\dbR^n)$ at $t=t_0$ to $x_1\in L^2_{\cF_{t_1}}(\Om;\dbR^n)$ at $t=t_1$} if
$$x^{t_0,x_0,u}(t_1)=x_1$$
almost surely. We then also say that {\it $u$ transfers $(t_0,x_0)$ to $(t_1,x_1)$},
or that {\it $(t_1,x_1)$ can be reached from $(t_0,x_0)$ by $u$}.
\end{definition}

\begin{definition}\label{def:controllability}\rm
System \rf{[ACBD]} is called {\it exactly controllable on $[t_0,t_1]$}, if for any
$x_0\in L^2_{\cF_{t_0}}(\Om;\dbR^n)$ and any $x_1\in L^2_{\cF_{t_1}}(\Om;\dbR^n)$
there exists a control $u\in L_\dbF^2(t_0,t_1;\dbR^m)$ transferring $(t_0,x_0)$ to $(t_1,x_1)$.
%System \rf{[ACBD]} is just called {\it exactly controllable} if it is exactly
%controllable on any subinterval $[t_0,t_1]$ of $[0,T]$.
\end{definition}

%\begin{definition}\rm
%Let $F\in\dbR^{k\times n}$. We say that a control $u(\cd)\in L_\dbF^2(t_0,t_1;\dbR^m)$
%{\it transfers the output of the system
%%
%$$dX(s)=\big[A(s)X(s)+B(s)u(s)\big]ds+\big[C(s)X(s)+D(s)u(s)\big]dW(s),\q Y(t)=FX(t)$$
%%
%to $y_1$ at $t=t_1$ from initial pair $(t_0,x_0)$} if
%%
%$$FX(t_1;t_0,x_0,u(\cd))=y_1.$$
%%
%\end{definition}

It was shown in \cite{Peng 1994} and \cite{Liu-Peng 2010} that system \rf{[ACBD]}
is exactly controllable on some interval only if $D$ has full row rank and the number
of columns of $D$ is greater than the number of rows of $D$ (i.e., $m>n$).
Note that $\rank(D)=n$ means that $DD^\top$ is invertible.
For technical reasons, in the sequel we shall impose, in addition to $m>n$, the following
slightly stronger condition (which is usually referred to as the {\it nondegeneracy} condition):
for some $\d>0$,
\bel{rank(D)=n} D(s)D(s)^\top \ges\d I_n, \q\forall s\in[0,T]. \ee
This condition implies that we can find a bounded invertible function $M:[0,T]\to\dbR^{m\times m}$ such that
\bel{DM=} D(s)M(s)=(I_n,0_{n\times (m-n)}), \q\forall s\in[0,T].\ee
In order to study the controllability of system \rf{[ACBD]},
we write $B(s)M(s)=(K(s),L(s))$ with $K(s)$ and $L(s)$ taking values in $\dbR^{n\times n}$ and
$\dbR^{n\times(m-n)}$, respectively, and introduce the following controlled system:
\bel{sys-barX} d\bar x(s) = \big[\bar A(s)\bar x(s)+\bar B(s)\bar u(s)\big]ds + \bar D(s)\bar u(s)dW(s),\ee
where
\bel{barABD} \bar A=A-KC, \q \bar B=BM=(K,L), \q \bar D=DM=(I_n,0_{n\times (m-n)}). \ee
Note that if we write the control $\bar u$ as the form
$$\bar u(s) = \begin{pmatrix}z(s) \\ v(s)\end{pmatrix}; \q z(s)\in\dbR^n,~v(s)\in\dbR^{m-n},$$
the system \rf{sys-barX} simplifies to
\bel{[ACBD]*} d\bar x(s) = \big[\bar A(s)\bar x(s)+K(s)z(s)+L(s)v(s)\big]ds + z(s)dW(s).\ee

The following result establishes a connection between the controllability of systems \rf{[ACBD]} and \rf{[ACBD]*}.

\begin{proposition}\label{prop:[ACBD]-[ACBD]*}
Let $0\les t_0<t_1\les T$, $x_0\in L^2_{\cF_{t_0}}(\Om;\dbR^n)$ and $x_1\in L^2_{\cF_{t_1}}(\Om;\dbR^n)$.
For system \rf{[ACBD]*}, a control $(z,v)\in L_\dbF^2(t_0,t_1;\dbR^n)\times L_\dbF^2(t_0,t_1;\dbR^{m-n})$
transfers $(t_0,x_0)$ to $(t_1,x_1)$ if and only if the control defined by
$$u(s)\deq M(s)\begin{pmatrix} z(s)-C(s)\bar x(s) \\ v(s)\end{pmatrix},\q s\in[t_0,t_1]$$
does so for system \rf{[ACBD]}, where $\bar x$ is the solution of \rf{[ACBD]*} with initial state $x_0$.
\end{proposition}

\begin{proof}
We first observe that
\begin{alignat*}{2}
\bar A(s)\bar x(s)+K(s)z(s)+L(s)v(s)
 &=  A(s)\bar x(s)+K(s)[z(s)-C(s)\bar x(s)]+L(s)v(s)\\
 &=  A(s)\bar x(s)+B(s)M(s)\begin{pmatrix} z(s)-C(s)\bar x(s)\\v(s)\end{pmatrix}\\
 &=  A(s)\bar x(s)+B(s)u(s),\\
C(s)\bar x(s)+D(s)u(s)
 &= C(s)\bar x(s)+D(s)M(s)\begin{pmatrix}z(s)-C(s)\bar x(s)\\v(s)\end{pmatrix}\\
 &= C(s)\bar x(s)+(I_n,0_{n\times (m-n)})\begin{pmatrix}z(s)-C(s)\bar x(s)\\v(s)\end{pmatrix}\\
 &= z(s).
\end{alignat*}
This means $\bar x$ also satisfies
$$d\bar x(s)=[A(s)\bar x(s)+B(s)u(s)]ds + [C(s)\bar x(s)+D(s)u(s)]dW(s).$$
Thus, by the uniqueness of a solution, with the initial state $x_0$ and the control $u$,
the solution $x$ of system \rf{[ACBD]} coincides with $\bar x$.
The result then follows immediately.
\end{proof}

From \autoref{prop:[ACBD]-[ACBD]*}, we see that the controllability of system \rf{[ACBD]}
is equivalent to that of system \rf{[ACBD]*}.
For the controllability of system \rf{[ACBD]*}, we have the following characterization,
which refines the result of Liu and Peng \cite[Theorem 2]{Liu-Peng 2010}.

\begin{proposition}\label{prop:controllability}
Let $0\les t_0<t_1\les T$, $x_0\in L^2_{\cF_{t_0}}(\Om;\dbR^n)$ and $x_1\in L^2_{\cF_{t_1}}(\Om;\dbR^n)$.
There exists a control $(z,v)\in L_\dbF^2(t_0,t_1;\dbR^n)\times L_\dbF^2(t_0,t_1;\dbR^{m-n})$ which
transfers the state of system \rf{[ACBD]*} from $x_0$ at $t=t_0$ to $x_1$ at $t=t_1$ if and only if
$x_0-\dbE[\F(t_0,t_1)x_1|\cF_{t_0}]$ belongs to the range space of
\bel{Grammian} \Psi(t_0,t_1) \deq \dbE\lt[\int_{t_0}^{t_1}\F(t_0,s)L(s)[\F(t_0,s)L(s)]^\top ds\rt] \ee
almost surely, that is, there exists an $\xi\in L^2_{\cF_{t_0}}(\Om;\dbR^n)$ such that
$$x_0-\dbE[\F(t_0,t_1)x_1|\cF_{t_0}]=\Psi(t_0,t_1)\xi,\q\as,$$
where $\F(t,s)=\F(t)^{-1}\F(s)$ with $\F=\{\F(s);0\les s\les T\}$ being the
solution to the following SDE for $\dbR^{n\times n}$-valued processes:
\bel{eqn:Phi}\left\{\begin{aligned}
d\F(s)  &= -\F(s)\bar A(s)ds-\F(s)K(s)dW(s), \q s\in[0,T],\\
 \F(0)  &= I_n.
\end{aligned}\right.\ee
\end{proposition}

\begin{proof}
{\it Sufficiency}. Suppose that there exists an $\xi\in L^2_{\cF_{t_0}}(\Om;\dbR^n)$ such that 
$$x_0-\dbE[\F(t_0,t_1)x_1|\cF_{t_0}]=\Psi(t_0,t_1)\xi,\q\as$$
Define
$$v(s)=-[\F(t_0,s)L(s)]^\top\xi,\q s\in[t_0,t_1],$$
and let $(y_1,z_1)$ be the adapted solution to the following BSDE:
$$\left\{\begin{aligned}
 dy_1(s)  &= \big[\bar A(s)y_1(s)+K(s)z_1(s)+L(s)v(s)\big]ds+z_1(s)dW(s), \q s\in[t_0,t_1],\\
y_1(t_1)  &= 0.
\end{aligned}\right.$$
According to \autoref{lemma:Y-formula},
\begin{align*}
y_1(t_0)
&=  -\dbE\lt[\int_{t_0}^{t_1}\F(t_0,s)L(s)v(s)ds\bigg|\cF_{t_0}\rt]\\
&=  \dbE\lt[\int_{t_0}^{t_1}\F(t_0,s)L(s)[\F(t_0,s)L(s)]^\top\xi ds\bigg|\cF_{t_0}\rt].
\end{align*}
Noting that $\xi$ is $\cF_{t_0}$-measurable and that $\F(t_0,s)$ is independent of $\cF_{t_0}$ for $s\ges t$,
we further obtain
\begin{align*}
y_1(t_0) &= \dbE\lt[\int_{t_0}^{t_1}\F(t_0,s)L(s)[\F(t_0,s)L(s)]^\top ds\rt]\xi = \Psi(t_0,t_1)\xi \\
         &= x_0-\dbE[\F(t_0,t_1)x_1|\cF_{t_0}].
\end{align*}
%
%By  It\^{o}'s formula,
%%
%$$ -\F(t_0)y_1(t_0) = \int_{t_0}^{t_1}\F(s)L(s)v(s)ds+\int_{t_0}^{t_1}\F(s)[z_1(s)-K(s)y_1(s)]dW(s).$$
%%
%Taking conditional expectations with respect to $\cF_{t_0}$ on both sides of the above, we obtain
%%
%$$-\F(t_0)y_1(t_0)=\dbE\lt[\int_{t_0}^{t_1}\F(s)L(s)v(s)ds\bigg|\cF_{t_0}\rt].$$
%%
%Using the fact that $\F(t_0)$ and $\xi$ are $\cF_{t_0}$-measurable, we have from the above that
%%
%\begin{align*}
%y_1(t_0)
%&=  -\dbE\lt[\int_{t_0}^{t_1}\F(t_0,s)L(s)v(s)ds\bigg|\cF_{t_0}\rt]\\
%      %
%&=  \dbE\lt[\int_{t_0}^{t_1}\F(t_0,s)L(s)[\F(t_0,s)L(s)]^\top\xi ds\bigg|\cF_{t_0}\rt]\\
%      %
%&=  \dbE\lt[\int_{t_0}^{t_1}\F(t_0,s)L(s)[\F(t_0,s)L(s)]^\top ds\bigg|\cF_{t_0}\rt]\xi\\
%      %
%&=  x_0-\dbE[\F(t_0,t_1)x_1|\cF_{t_0}].
%\end{align*}
%
Now let $(y_2,z_2)$ be the adapted solution to the BSDE
\bel{eqn:Y2&z2}\left\{\begin{aligned}
 dy_2(s) &= \big[\bar A(s)y_2(s)+K(s)z_2(s)\big]ds + z_2(s)dW(s), \q s\in[t_0,t_1],\\
y_2(t_1) &= x_1,
\end{aligned}\right.\ee
and define
$$\bar x(s)=y_1(s)+y_2(s),\q z(s)=z_1(s)+z_2(s),\q s\in[t_0,t_1].$$
By \autoref{lemma:Y-formula}, $y_2(t_0)=\dbE[\F(t_0,t_1)x_1|\cF_{t_0}]$,
and thus, by linearity, $(\bar x,z,v)$ satisfies
$$\left\{\begin{aligned}
 d\bar x(s) &= \big[\bar A(s)\bar x(s)+K(s)z(s)+L(s)v(s)\big]ds + z(s)dW(s), \q s\in[t_0,t_1],\\
\bar x(t_0) &= x_0, \q \bar x(t_1)=x_1.
\end{aligned}\right.$$
This shows $(t_1,x_1)$ can be reached from $(t_0,x_0)$ by $(z,v)$.

\ms

{\it Necessity}. We prove the necessity by contradiction.
Suppose that $(t_1,x_1)$ can be reached from $(t_0,x_0)$ by some control $(z,v)$
but there exists some $\Om^\prime\subseteq\Om$ with $\dbP(\Om^\prime)>0$ such that
$x_0(\om)-\dbE[\F(t_0,t_1)x_1|\cF_{t_0}](\om)$ does not lie in the range space
of $\Psi(t_0,t_1)$ for every $\om\in\Om^\prime$.
Then we can find an $\b\in L^2_{\cF_{t_0}}(\Om;\dbR^n)$ such that
$$\Psi(t_0,t_1)\b=0,~\as,\q\hb{and}\q\dbE\big(\b^\top\b_0\big)>0,$$
where $\b_0=x_0-\dbE[\F(t_0,t_1)x_1|\cF_{t_0}]$.
Let $\bar x$ be the corresponding state process.
By applying the integration by parts formula to $\F\bar x$, we have
$$\F(t_1)x_1-\F(t_0)x_0 =\int_{t_0}^{t_1}\F(s)L(s)v(s)ds + \int_{t_0}^{t_1}\F(s)\big[z(s)-K(s)\bar x(s)\big]dW(s).$$
Taking conditional expectations with respect to $\cF_{t_0}$ on both sides of the above, we get
$$-\F(t_0)\b_0 = \dbE[\F(t_1)x_1|\cF_{t_0}]-\F(t_0)x_0
                   = \dbE\lt[\int_{t_0}^{t_1}\F(s)L(s)v(s)ds\bigg|\cF_{t_0}\rt], $$
from which it follows that
\begin{align}\label{contradiction>0}
0  &< \dbE\big(\b^\top\b_0\big)
    = -\dbE\lt(\b^\top\F(t_0)^{-1}\dbE\lt[\int_{t_0}^{t_1}\F(s)L(s)v(s)ds\bigg|\cF_{t_0}\rt]\rt) \nn\\
&=  -\dbE\lt(\b^\top\F(t_0)^{-1}\int_{t_0}^{t_1}\F(s)L(s)v(s)ds\rt) \nn\\
&= -\dbE\lt(\int_{t_0}^{t_1}\b^\top\F(t_0,s)L(s)v(s)ds\rt).
\end{align}
But, using the fact that $\Psi(t_0,t_1)\b=0$ $\as$ and
noting that $\b$ is independent of $\F(t_0,s)$ for $s\ges t_0$, we have
\begin{align*}
0 &= \dbE\big(\b^\top\Psi(t_0,t_1)\b\big)
   = \dbE\int_{t_0}^{t_1}\b^\top\F(t_0,s)L(s)[\F(t_0,s)L(s)]^\top\b ds\\
  &= \dbE\int_{t_0}^{t_1}\big|\b^\top\F(t_0,s)L(s)\big|^2 ds,
\end{align*}
which implies the vanishing of $\b^\top\F(t_0,s)L(s)$ and the contradiction of \rf{contradiction>0}.
\end{proof}

\begin{remark}\rm
The matrix $\Psi(t_0,t_1)$ defined by \rf{Grammian} is called the {\it controllability Gramian}
of system \rf{[ACBD]*} over $[t_0,t_1]$.
Note that $\Psi(t_0,t_1)$ is symmetric positive semidefinite.
\end{remark}

Propositions \ref{prop:[ACBD]-[ACBD]*} and \ref{prop:controllability}
have some easy consequences which we summarize as follows.

\begin{corollary}\label{coro:controllability}
Let $0\les t_0<t_1\les T$, and let $\F$ be the solution to \rf{eqn:Phi}.
\begin{enumerate}[\indent\rm(i)]
\item System \rf{[ACBD]} is exactly controllable on $[t_0,t_1]$ if and only if system \rf{[ACBD]*} is so.
\item System \rf{[ACBD]*} is exactly controllable on $[t_0,t_1]$ if and only if the controllability Gramian
      $\Psi(t_0,t_1)$ is positive definite.
\item Let $F\in\dbR^{k\times n}$ and $b\in L^2_{\cF_{t_0}}(\Om;\dbR^k)$.
      For system \rf{[ACBD]*}, there exists a point on the stochastic linear manifold
      $$\cH(F,b) = \{\xi\in L^2_{\cF_{t_0}}(\Om;\dbR^n): F\xi=b\}$$
      that can be transferred to $x_1\in L^2_{\cF_{t_1}}(\Om;\dbR^n)$ at time $t=t_1$
      if and only if there exist an $\xi\in L^2_{\cF_{t_0}}(\Om;\dbR^n)$ such that
      $$F\Psi(t_0,t_1)\xi = b-F\dbE[\F(t_0,t_1)x_1|\cF_{t_0}].$$
\end{enumerate}
\end{corollary}

\begin{proof}
(i) It is a direct consequence of \autoref{prop:[ACBD]-[ACBD]*}.

\ms

(ii) If $\Psi(t_0,t_1)>0$, then obviously, $x_0-\dbE[\F(t_0,t_1)x_1|\cF_{t_0}]$ belongs to $\sR(\Psi(t_0,t_1))$,
the range of $\Psi(t_0,t_1)$, for all $x_0\in L^2_{\cF_{t_0}}(\Om;\dbR^n)$
and all $x_1\in L^2_{\cF_{t_1}}(\Om;\dbR^n)$.
Thus, by \autoref{prop:controllability}, system \rf{[ACBD]*} is exactly controllable on $[t_0,t_1]$.
Conversely, if system \rf{[ACBD]*} is exactly controllable on $[t_0,t_1]$,
then for $x_1=0$ and any $x_0\in\dbR^n$,
$$x_0=x_0-\dbE[\F(t_0,t_1)x_1|\cF_{t_0}]\in\sR(\Psi(t_0,t_1)),$$
which implies that $\Psi(t_0,t_1)$ has full rank and hence is positive definite.

\ms

(iii) By \autoref{prop:controllability} we know that a state $x_1\in L^2_{\cF_{t_1}}(\Om;\dbR^n)$
can be reached at $t_1$ from some $x_0\in\cH(F,b)$ if and only if there exists an
$\xi\in L^2_{\cF_{t_0}}(\Om;\dbR^n)$ such that
$$\Psi(t_0,t_1)\xi = x_0-\dbE[\F(t_0,t_1)x_1|\cF_{t_0}].$$
Thus, the state $x_1$ can be reached from $\cH(F,b)$ if and only if the $\xi$ is such that
$$F\{\Psi(t_0,t_1)\xi +\dbE[\F(t_0,t_1)x_1|\cF_{t_0}]\} =b.$$
The desired result then follows readily.
\end{proof}

The construction in the proof of \autoref{prop:controllability} actually provides an explicit procedure
for finding a control that accomplishes desired transfers. Let us recap and conclude this section.

\begin{proposition}\label{prop:candidate}
Let $0\les t_0<t_1\les T$ and $x_1\in L^2_{\cF_{t_1}}(\Om;\dbR^n)$.
Let $F\in\dbR^{k\times n}$ and $b\in L^2_{\cF_{t_0}}(\Om;\dbR^k)$.
If $\xi\in L^2_{\cF_{t_0}}(\Om;\dbR^n)$ is such that
$$F\Psi(t_0,t_1)\xi = b-F\dbE[\F(t_0,t_1)x_1|\cF_{t_0}],$$
then with
$$v(s) \deq -L(s)^\top\F(t_0,s)^\top\xi,\q s\in[t_0,t_1],$$
and $z=\{z(s);t_0\les s\les t_1\}$ being the second component of the adapted solution to the BSDE
$$\left\{\begin{aligned}
dy(s)  &= \big[\bar A(s)y(s)+K(s)z(s)+L(s)v(s)\big]ds + z(s)dW(s),\q s\in[t_0,t_1],\\
y(t_1) &= x_1,
\end{aligned}\right.$$
$(z,v)$ transfers the state of the system \rf{[ACBD]*} from
$$x_0 = \Psi(t_0,t_1)\xi + \dbE[\F(t_0,t_1)x_1|\cF_{t_0}] \in\cH(F,b)$$
at $t=t_0$ to $x_1$ at $t=t_1$.
\end{proposition}

\section{Lagrange multipliers and unconstrained backward LQ problems}\label{Sec:Lagrange}

We now return to\autoref{Prob}. Recall that the nondegeneracy condition \rf{rank(D)=n} is assumed
so that the target $\eta$ can be reached from a given stochastic linear manifold $\cH(F,b)$.
Let $M$ be as in \rf{DM=} and $\bar A$, $K$, $L$ be as in \rf{barABD}.
We have seen from \autoref{prop:[ACBD]-[ACBD]*} that systems \rf{[ACBD]} and \rf{[ACBD]*} share the
same controllability.
So by appropriate transformations, we may assume without loss of generality that the state equation
\rf{state} takes the form
\bel{state*} dx(s) = [A(s)x(s)+K(s)z(s)+L(s)v(s)]ds + z(s)dW(s), \q s\in[t,T],\ee
and that the cost functional \rf{cost} takes the following form:
\begin{align}\label{cost*}
J(t,z,v) &= \dbE\bigg\{\lan Gx(t),x(t)\ran +\int_t^T\[\lan Q(s)x(s),x(s)\ran  \nn\\
         &\hp{=\dbE\bigg\{\q} + \lan R(s)z(s),z(s)\ran + \lan N(s)v(s),v(s)\ran\]ds\bigg\}.
\end{align}
That is, the coefficients $B$ and $D$ of \rf{state} are given by
$$ B(s) = (K(s),L(s)), \q D(s) = (I_n,0_{n\times (m-n)}); \q s\in[0,T], $$
and the control $u$ is $\begin{pmatrix}z \\ v\end{pmatrix}$.
In this case, with the given terminal state $\eta$, we may think of $v$ alone as the control and
regard $(x,z)$ as the adapted solution to the BSDE
\bel{B-state}\left\{\begin{aligned}
dx(s) &= [A(s)x(s)+K(s)z(s)+L(s)v(s)]ds + z(s)dW(s), \q s\in[t,T],\\
 x(T) &= \eta.
\end{aligned}\right.\ee
Further, since for given $\eta$, $z$ is uniquely decided by $v$,
we can simply write the cost functional \rf{cost*} as $J(t,v)$.
Therefore, solving\autoref{Prob} is equivalent to finding an optimal control $v^*$ for the following
constrained backward LQ problem.

\begin{taggedthm}{Problem (CBLQ)}\label{CBLQ}
For a given terminal state $\eta\in L^2_{\cF_T}(\Om;\dbR^n)$, find a control
$v^*\in L_\dbF^2(t,T;\dbR^{m-n})$ such that the corresponding adapted solution
$(x^*,z^*)$ of \rf{B-state} satisfies $x^*(t)\in\cH(F,b)$, and
\bel{mincJ} J(t,v^*)\les J(t,v), \q\forall v\in L_\dbF^2(t,T;\dbR^{m-n}).\ee
\end{taggedthm}

For this reduced problem, we impose the following assumptions that are similar to
the conditions \ref{(A1)} and \ref{(A2)}.
\begin{itemize}
\item[{\setword{(H1)}{(H1)}}]
$A,K:[0,T]\to\dbR^{n\times n}$ and $L:[0,T]\to\dbR^{n\times(m-n)}$ are bounded measurable functions.
\item[{\setword{(H2)}{(H2)}}]
$G$ is a symmetric $n\times n$ matrix;
$Q,R:[0,T]\to\dbR^{n\times n}$ and  $N:[0,T]\to\dbR^{(m-n)\times(m-n)}$
are bounded and symmetric. Moreover, for some $\d>0$,
$$G\ges 0, \q Q(s)\ges0, \q R(s)\ges0, \q N(s)\ges\d I_{m-n}, \q\ae~s\in[0,T].$$
\end{itemize}

To find an optimal control for\autoref{CBLQ}, let $\l\in L^2_{\cF_t}(\Om;\dbR^k)$ be undetermined and define
\begin{align}\label{cost-l}
J_\l(t,v) &\deq J(t,v) + 2\dbE\lan F^\top\l,x(t)\ran \nn\\
          &=   \dbE\bigg\{\lan Gx(t),x(t)\ran + 2\dbE\lan F^\top\l,x(t)\ran +\int_t^T\[\lan Q(s)x(s),x(s)\ran \nn\\
          &\hp{=\qq} +\lan R(s)z(s),z(s)\ran + \lan N(s)v(s),v(s)\ran\]ds\bigg\}.
\end{align}
Consider the following parameterized unconstrained backward LQ problem.

\begin{taggedthm}{Problem (BLQ)$_\l$}\label{BLQ}
For a given terminal state $\eta\in L^2_{\cF_T}(\Om;\dbR^n)$, find a control
$v^*\in L_\dbF^2(t,T;\dbR^{m-n})$ such that
\bel{minJl} J_\l(t,v^*)\les J_\l(t,v), \q\forall v\in L_\dbF^2(t,T;\dbR^{m-n}),\ee
subject to the backward state equation \rf{B-state}.
\end{taggedthm}

If for some parameter $\l\in L^2_{\cF_t}(\Om;\dbR^k)$, the optimal control $v^*_\l$ of\autoref{BLQ}
is such that the initial state of system \rf{B-state} falls on the stochastic linear manifold $\cH(F,b)$,
then intuitively we can convince ourselves that $v^*_\l$ is also optimal for\autoref{CBLQ}.
In fact, we have the following result.

\begin{proposition}\label{prop:relation}
Let {\rm\ref{(H1)}--\ref{(H2)}} hold, and let $\eta\in L^2_{\cF_T}(\Om;\dbR^n)$ be given.
If $v^*_\l$ is an optimal control of\autoref{BLQ} such that the adapted solution $(x^*_\l,z^*_\l)$ of
$$\left\{\begin{aligned}
 dx^*_\l(s)  &= [A(s)x^*_\l(s)+K(s)z^*_\l(s)+L(s)v^*_\l(s)]ds + z^*_\l(s)dW(s), \q s\in[t,T],\\
  x^*_\l(T)  &= \eta,
\end{aligned}\right.$$
satisfies $x^*_\l(t)\in\cH(F,b)$, then $v^*_\l$ is also optimal for\autoref{CBLQ}.
\end{proposition}

\begin{proof}
Since $v^*_\l$ is optimal for\autoref{BLQ}, \rf{minJl} holds.
In particular, for any $v\in L_\dbF^2(t,T;\dbR^{m-n})$ such that the initial state
of system \rf{B-state} falls on $\cH(F,b)$, we have
\begin{align*}
& J(t,v^*_\l)+2\dbE\lan F^\top\l,x^*_\l(t)\ran = J_\l(t,v^*_\l)\les J_\l(t,v) = J(t,v)+2\dbE\lan F^\top\l,x(t)\ran, \\
& Fx(t)=Fx^*_\l(t)=b,
\end{align*}
from which it follows that
\begin{align*}
J(t,v^*_\l) &\les J(t,v)+2\dbE\lan F^\top\l,x(t)-x^*_\l(t)\ran \\
            &= J(t,v)+2\dbE\lan \l,F[x(t)-x^*_\l(t)]\ran = J(t,v).
\end{align*}
This completes the proof.
\end{proof}

According to \autoref{prop:relation}, the procedure for finding the optimal control of
our original\autoref{CBLQ} can be divided into two steps.
\begin{enumerate}[\indent Step 1.]
\item Construct the optimal control $v^*_\l$ for the parameterized unconstrained backward LQ problem.
\item Select the parameter $\l$ such that the corresponding optimal state process $x^*_\l$ of\autoref{BLQ}
      satisfies $x^*_\l(t)\in\cH(F,b)$.
\end{enumerate}
For Step 1, we first present the following result,
which characterizes the optimal control of\autoref{BLQ} in terms of FBSDEs.

\begin{theorem}\label{thm:optmality}
Let {\rm\ref{(H1)}--\ref{(H2)}} hold.
Let $\l\in L^2_{\cF_t}(\Om;\dbR^k)$ and $\eta\in L^2_{\cF_T}(\Om;\dbR^n)$ be given.
Then a control $v^*\in L_\dbF^2(t,T;\dbR^{m-n})$ is optimal for\autoref{BLQ}
if and only if the adapted solution $(x^*,z^*,y^*)$ to the coupled FBSDE
\bel{xzy-star}\left\{\begin{aligned}
  dx^*(s) &= (Ax^*+Kz^*+Lv^*)ds + z^*dW(s), \\
  dy^*(s) &= (-A^\top y^* + Qx^*)ds + (-K^\top y^* +Rz^*)dW(s), \\
   x^*(T) &= \eta, \q y^*(t)= Gx^*(t)+F^\top\l,
\end{aligned}\right.\ee
satisfies the following stationarity condition:
\bel{stationarity} Nv^*-L^\top y^*=0, \q\ae~\hbox{on}~[t,T],~\as \ee
\end{theorem}

\begin{proof}
First note that $v^*$ is optimal if and only if
\begin{equation}\label{3.14-1}
J_\l(t,v^*+\e v) -J_\l(t,v^*)\ges0, \q\forall \e\in\dbR,~\forall v\in L_\dbF^2(t,T;\dbR^{m-n}).
\end{equation}
For fixed but arbitrary $\e\in\dbR$ and $v\in L_\dbF^2(t,T;\dbR^{m-n})$, we have by linearity that
the adapted solution $(x_\e,z_\e)$ to
$$\left\{\begin{aligned}
  dx_\e(s) &= [Ax_\e+Kz_\e+L(v^*+\e v)]ds + z_\e dW(s), \q s\in[t,T],\\
   x_\e(T) &= \eta,
\end{aligned}\right.$$
is the sum of $(x^*,z^*)$ and $\e(x,z)$, where $(x,z)$ is the adapted solution to
$$\left\{\begin{aligned}
  dx(s) &= (Ax+Kz+Lv)ds + z dW(s), \q s\in[t,T],\\
   x(T) &= 0.
\end{aligned}\right.$$
Then it follows by a straightforward computation that
\begin{align*}
J_\l(t,v^*+\e v) &= \e^2\dbE\bigg[\lan Gx(t),x(t)\ran
                    +\int_t^T\(\lan Qx,x\ran + \lan Rz,z\ran + \lan Nv,v\ran\)ds\bigg] \\
                 &\hp{=\ } +2\e\dbE\bigg[\lan Gx^*(t)+F^\top\l,x(t)\ran
                    +\int_t^T\(\lan Qx^*,x\ran + \lan Rz^*,z\ran + \lan Nv^*,v\ran\)ds\bigg] \\
                 &\hp{=\ } +J_\l(t,v^*).
\end{align*}
Thus, \rf{3.14-1} in turn is equivalent to
\begin{align}\label{3.14-2}
& \e^2\dbE\bigg[\lan Gx(t),x(t)\ran
    +\int_t^T\(\lan Qx,x\ran + \lan Rz,z\ran + \lan Nv,v\ran\)ds\bigg] \nn\\
&~ +2\e\dbE\bigg[\lan Gx^*(t)+F^\top\l,x(t)\ran
    +\int_t^T\(\lan Qx^*,x\ran + \lan Rz^*,z\ran + \lan Nv^*,v\ran\)ds\bigg]\ges0
\end{align}
for all $\e\in\dbR$ and all $v\in L_\dbF^2(t,T;\dbR^{m-n})$.
Since the term in the first square bracket is nonnegative by the assumption \ref{(H2)},
\rf{3.14-2} holds for all $\e\in\dbR$ if and only if
%for any $v\in L_\dbF^2(t,T;\dbR^{m-n})$,
%
\begin{equation}\label{3.14-3}
\dbE\bigg[\lan Gx^*(t)+F^\top\l,x(t)\ran
    +\int_t^T\(\lan Qx^*,x\ran + \lan Rz^*,z\ran + \lan Nv^*,v\ran\)ds\bigg]=0.
\end{equation}
Now by applying It\^{o}'s rule to $s\mapsto\lan y^*(s),x(s)\ran$, we obtain
\begin{align*}
\dbE\lan Gx^*(t)+F^\top\l,x(t)\ran =\dbE\lan y^*(t),x(t)\ran
  = -\dbE\int_t^T\(\lan Qx^*,x\ran+\lan L^\top y^*,v\ran+\lan Rz^*,z\ran\)ds,
\end{align*}
substituting which into \rf{3.14-3} yields
$$\dbE\int_t^T\lan Nv^*- L^\top y^*,v\ran ds=0.$$
Since the above has to be true for all $v\in L_\dbF^2(t,T;\dbR^{m-n})$, \rf{stationarity} follows.
The sufficiency of \rf{stationarity} can be proved by reversing the above argument.
\end{proof}

We call \rf{xzy-star}, together with the stationarity condition \rf{stationarity},
the {\it optimality system} for\autoref{BLQ}.
Note that from \rf{stationarity} we can represent the optimal control $v^*$ in terms
of $y^*$ as $v^*= N^{-1}L^\top y^*$.
Substituting for $v^*$ then brings a coupling into the FBSDE \rf{xzy-star}.
So in order to find the optimal control $v^*$, one actually need to solve a coupled FBSDE.

\medskip

To construct an optimal control for\autoref{BLQ} from the optimality system
\rf{xzy-star}-\rf{stationarity}, we now introduce the following Riccati-type equation:
\bel{Ric-Sigma}\left\{\begin{aligned}
  & \dot\Si-\Si A^\top-A\Si-\Si Q\Si+LN^{-1}L^\top+K(I_n+\Si R)^{-1}\Si K^\top=0, \q s\in[0,T],\\
  & \Si(T)=0.
\end{aligned}\right.\ee
It was shown in \cite{Lim-Zhou 2001} (see also \cite{Li-Sun-Xiong 2017} for an alternative proof)
that equation \rf{Ric-Sigma} has a unique positive semidefinite solution $\Si\in C([0,T];\dbS^n)$:
$$\Si(s)^\top=\Si(s), \q \Si(s)\ges0; \q\forall s\in[0,T].$$
This allows us to consider the following BSDE:
\bel{BSDE-f}\left\{\begin{aligned}
  d\f(s) &= [(A+\Si Q)\f +K(I_n+\Si R)^{-1}\b]ds + \b dW(s), \q s\in[0,T],\\
   \f(T) &= -\eta,
\end{aligned}\right.\ee
which, by the standard result for BSDEs, admits a unique adapted solution
$$(\f,\b)\in L_\dbF^2(\Om;C([0,T];\dbR^n))\times L_\dbF^2(0,T;\dbR^n).$$
Consider further the following $(\f,\b,\l)$-dependent SDE:
\bel{SDE-y}\left\{\begin{aligned}
dy(s) &= -[(A^\top\1n+Q\Si)y+Q\f]ds -(I_n+R\Si)^{-1}(K^\top y+R\b)dW(s), ~s\in[t,T],\\
 y(t) &= [I_n+G\Si(t)]^{-1}[F^\top\l-G\f(t)].
\end{aligned}\right.\ee
Obviously, \rf{SDE-y} is uniquely solvable.

\begin{theorem}\label{thm:solution-BLQ-lamda}
Let {\rm\ref{(H1)}--\ref{(H2)}} hold. Then\autoref{BLQ} admits a unique optimal control which is given by
$$v^*_\l(s)=N(s)^{-1}L(s)^\top y(s), \q s\in[t,T],$$
where $y$ is the solution to the SDE \rf{SDE-y}.
\end{theorem}

\begin{proof}
Let $(x,z)$ be the adapted solution to the BSDE
$$\left\{\begin{aligned}
  dx(s) &= (Ax+Kz+Lv^*_\l)ds + zdW(s), \q s\in[t,T],\\
   x(T) &= \eta.
\end{aligned}\right.$$
According to \autoref{thm:optmality}, it suffices to verify that the solution $y$ of \rf{SDE-y} satisfies the SDE
$$\left\{\begin{aligned}
  dy(s) &= (-A^\top y + Qx)ds + (-K^\top y+Rz)dW(s), \q s\in[t,T],\\
   y(t) &= Gx(t)+F^\top\l.
\end{aligned}\right.$$
This can be accomplished if we are able to show that
\begin{equation}\label{Decoupling}
x(s) = -[\Si(s)y(s)+\f(s)], \q z(s) = [I_n+\Si(s) R(s)]^{-1}[\Si(s) K(s)^\top y(s)-\b(s)].
\end{equation}
Indeed, if \rf{Decoupling} holds, then the first relation gives
$$Gx(t)+F^\top\l=-G\Si(t)y(t)-G\f(t)+F^\top\l,$$
which, together with the initial condition in \rf{SDE-y}, implies that
\begin{align*}
y(t) &= -G\Si(t)y(t)+[I_n+G\Si(t)]y(t)=-G\Si(t)y(t)+F^\top\l-G\f(t) \\
     &= Gx(t)+F^\top\l.
\end{align*}
Furthermore,
$$-[(A^\top\1n+Q\Si)y + Q\f] = -A^\top y + Qx,$$
and using the second relation in \rf{Decoupling} we obtain
\begin{align*}
K^\top y -(I_n+R\Si)^{-1}(K^\top y+R\b)
  &= [I_n-(I_n+R\Si)^{-1}]K^\top y -(I_n+R\Si)^{-1}R\b \\
  &= (I_n+R\Si)^{-1}R\Si K^\top y -(I_n+R\Si)^{-1}R\b \\
  &= (I_n+R\Si)^{-1}R(\Si K^\top y -\b) \\
  &= R(I_n+\Si R)^{-1}(\Si K^\top y -\b) \\
  &= Rz,
\end{align*}
and hence $-(I_n+R\Si)^{-1}(K^\top y+R\b)= -K^\top y+Rz$.

\medskip

In order to prove \rf{Decoupling}, let us denote
$$ \hat x(s) \deq -[\Si(s)y(s)+\f(s)], \q \hat z(s) \deq [I_n+\Si(s) R(s)]^{-1}[\Si(s) K(s)^\top y(s)-\b(s)].$$
Thanks to the uniqueness of an adapted solution, our proof will be complete if we can show that $(\hat x,\hat z)$
satisfies the same BSDE as $(x,z)$.
To this end, we first note that $\hat x(T)=-[\Si(T)y(T)+\f(T)]=\eta$. Moreover, by It\^{o}'s rule,
\begin{align}\label{eqn:hatx}
d\hat x &= d(-\Si y-\f)= -\dot\Si y ds -\Si dy -d\f \nn\\
        &= -\dot\Si y ds +\Si[(A^\top\1n+Q\Si)y+Q\f]ds +\Si(I_n+R\Si)^{-1}(K^\top y+R\b)dW \nn\\
        &\hp{=\ } -\1n[(A+\Si Q)\f +K(I_n+\Si R)^{-1}\b]ds - \b dW \nn\\
        &= [(-\dot\Si+ \Si A^\top +\Si Q\Si)y -A\f-K(I_n+\Si R)^{-1}\b]ds \nn\\
        &\hp{=\ } +\1n\{\Si(I_n+R\Si)^{-1}K^\top y + [\Si(I_n+R\Si)^{-1}R-I_n]\b\}dW.
\end{align}
Using \rf{Ric-Sigma}, we can rewrite the drift term in \rf{eqn:hatx} as
\begin{align*}
& (-\dot\Si+ \Si A^\top +\Si Q\Si)y -A\f-K(I_n+\Si R)^{-1}\b \\
&\q = [-A\Si + LN^{-1}L^\top +K(I_n+\Si R)^{-1}\Si K^\top]y -A\f-K(I_n+\Si R)^{-1}\b \\
&\q = -A(\Si y+\f) + LN^{-1}L^\top y + K(I_n+\Si R)^{-1}(\Si K^\top y-\b) \\
&\q = A\hat x + Lv^*_\l + K\hat z.
\end{align*}
Using the fact that
$$\Si(I_n+R\Si)^{-1} = (I_n+\Si R)^{-1}\Si, \q \Si(I_n+R\Si)^{-1}R-I_n = -(I_n+\Si R)^{-1},$$
we can rewrite the diffusion term in \rf{eqn:hatx} as
\begin{align*}
& \Si(I_n+R\Si)^{-1}K^\top y + [\Si(I_n+R\Si)^{-1}R-I_n]\b \\
&\q = (I_n+\Si R)^{-1}\Si K^\top y -(I_n+\Si R)^{-1}\b \\
&\q = \hat z.
\end{align*}
This shows that $(\hat x,\hat z)$ satisfies the same BSDE as $(x,z)$ and hence completes the proof.
\end{proof}

\section{Selection of optimal parameters}\label{Sec:Main-result}

In this section we show how to find a $\l\in L^2_{\cF_t}(\Om;\dbR^k)$, called an {\it optimal parameter},
such that the corresponding optimal state process of\autoref{BLQ} satisfies $x^*_\l(t)\in\cH(F,b)$.
It is worth pointing out that the usual method of Lagrange multipliers does not work efficiently in our situation,
due to the difficulty in computing the derivative of $J_\l(t,v^*_\l)$ in $\l$.
The key of our approach is to establish an equivalence relationship between the controllability of \rf{state*}
and a system involving $\Si$, the solution of the Riccati equation \rf{Ric-Sigma}.
It turns out that an optimal parameter exists and can be obtained by solving an algebraic equation.

\medskip

Recall that $\Si$ and $(\f,\b)$ are the unique solutions to equations \rf{Ric-Sigma} and \rf{BSDE-f}, respectively.
The main result of this section can be stated as follows.

\begin{theorem}\label{thm:main}
Let {\rm\ref{(H1)}--\ref{(H2)}} hold. If the state of system \rf{state*} can be transferred
to $(T,\eta)$ from the stochastic linear manifold $\cH(F,b)$, then the algebraic equation
\bel{Eqn-lamda} \Big\{F[I_n+\Si(t)G]^{-1}\Si(t)F^\top\Big\}\l = -\Big\{ F[I_n+\Si(t)G]^{-1}\f(t)+b \Big\}\ee
has a solution. Moreover, any solution $\l^*$ of \rf{Eqn-lamda} is an optimal optimal parameter,
and the optimal controls $v^*$ of\autoref{CBLQ} are given by
\begin{align*}
v^*(s) &= N(s)^{-1}L(s)^\top y^*(s), \q s\in[t,T],
\end{align*}
where $y^*$ is the solution of
\bel{eqn:y*}\left\{\begin{aligned}
dy^*(s) &= -[(A^\top\1n+Q\Si)y^*+Q\f]ds -(I_n+R\Si)^{-1}(K^\top y^*+R\b)dW, ~s\in[t,T],\\
 y^*(t) &= [I_n+G\Si(t)]^{-1}[F^\top\l^*-G\f(t)].
\end{aligned}\right.\ee
\end{theorem}

In preparation for the proof of \autoref{thm:main}, let us consider the following system:
\bel{hatx}d\hat x(s) = [\h A(s)\hat x(s)+\h K(s)\hat z(s)+\h L(s)\hat v(s)]ds+\hat z(s)dW(s),\ee
where the coefficients are given by
\begin{align}\label{hatAKL}
\h A = A+\Si Q, \q \h K= K(I_n+\Si R)^{-1}, \q \h L = (LN^{-{1\over2}},\, -\Si Q^{1\over2},\, K(I_n+\Si R)^{-1}\Si R^{1\over2}).
\end{align}
The following result shows that the controllability of system \rf{state*}
is equivalent to that of system \rf{hatx}.

\begin{proposition}\label{prop:x=hatx}
Let {\rm\ref{(H1)}--\ref{(H2)}} hold.
Let $0\les t_0<t_1\les T$, $x_0\in L^2_{\cF_{t_0}}(\Om;\dbR^n)$ and $x_1\in L^2_{\cF_{t_1}}(\Om;\dbR^n)$.
A control $(z,v)\in L_\dbF^2(t_0,t_1;\dbR^n)\times L_\dbF^2(t_0,t_1;\dbR^{m-n})$ transfers $(t_0,x_0)$
to $(t_1,x_1)$ for system \rf{state*} if and only if the control $(\hat z,\hat v)$ defined by
\bel{hatzv} \hat z(s)\deq z(s),
\q \hat v(s)\deq \begin{pmatrix}[N(s)]^{1\over2}v(s)\\ [Q(s)]^{1\over2}x(s) \\ [R(s)]^{1\over2}z(s)\end{pmatrix};
\q s\in[t_0,t_1] \ee
does so for system \rf{hatx}, where $x$ is the solution of \rf{state*}
with respect to the initial pair $(t_0,x_0)$ and the control $(z,v)$.
\end{proposition}

\begin{proof}
Let $(\hat z,\hat v)$ be defined by \rf{hatzv} and $\hat x$ be the solution to
\bel{hatx-5.21}\left\{\begin{aligned}
d\hat x(s)   &= [\h A(s)\hat x(s)+\h K(s)\hat z(s)+\h L(s)\hat v(s)]ds + \hat z(s)dW(s), \q s\in[t_0,t_1],\\
 \hat x(t_0) &= x_0.
\end{aligned}\right.\ee
We prove the assertion by showing $\hat x=x$. Substituting \rf{hatAKL} and \rf{hatzv} into \rf{hatx-5.21}, we have
\bel{hatx=x}\left\{\begin{aligned}
d\hat x(s)   &= [A\hat x+\Si Q(\hat x-x)+Kz+Lv]ds + z dW(s), \q s\in[t_0,t_1],\\
 \hat x(t_0) &= x_0.
\end{aligned}\right.\ee
Clearly, $x$ is also a solution of \rf{hatx=x} and hence $x=\hat x$ by the uniqueness of a solution.
\end{proof}

Although the system \rf{hatx} looks more complicated than \rf{state*},
the controllability Gramian of \rf{hatx} takes a simpler form, as shown by the following result.

\begin{proposition}\label{prop:Gramian-Si}
Let {\rm\ref{(H1)}--\ref{(H2)}} hold. For any $t\in[0,T]$,
the controllability Gramian of system \rf{hatx} over $[t,T]$ is $\Si(t)$.
\end{proposition}

\begin{proof}
Let $\Pi=\{\Pi(s);0\les s\les T\}$ be the solution to the following SDE for $\dbR^{n\times n}$-valued processes:
\bel{eqn:Pi}\lt\{\begin{aligned}
d\Pi(s)  &= -\Pi(s)\h A(s)ds-\Pi(s)\h K(s)dW(s),\q s\in[0,T],\\
 \Pi(0)  &= I_n,
\end{aligned}\rt.\ee
and let $\Pi(t,s)=\Pi(t)^{-1}\Pi(s)$. By \autoref{prop:controllability},
the controllability Gramian of system \rf{hatx} over $[t,T]$ is
$$\dbE\lt\{\int_t^T\Pi(t,s)\h L(s)\big[\Pi(t,s)\h L(s)\big]^\top ds\rt\}.$$
On the other hand, we have by It\^{o}'s rule that
\begin{align*}
d\(\Pi\Si\Pi^\top\)
  &= -\,\Pi(A+\Si Q)\Si\Pi^\top ds - \Pi K(I_n+\Si R)^{-1}\Si\Pi^\top dW \\
  &\hp{=\ } + \Pi\dot\Si\Pi^\top ds - \Pi\Si(A+\Si Q)^\top\Pi^\top ds - \Pi\Si(I_n+R\Si)^{-1}K^\top\Pi^\top dW \\
  &\hp{=\ } + \Pi K(I_n+\Si R)^{-1}\Si(I_n+R\Si)^{-1}K^\top\Pi^\top ds \\
  &= -\,\Pi\[(A+\Si Q)\Si - \dot\Si + \Si(A+\Si Q)^\top \\
  &\hp{=-\,\Pi\[~} - K(I_n+\Si R)^{-1}\Si(I_n+R\Si)^{-1}K^\top\]\Pi^\top ds \\
  &\hp{=\ } - \Pi\[K(I_n+\Si R)^{-1}\Si + \Si(I_n+R\Si)^{-1}K^\top\]\Pi^\top dW \\
  &= -\,\Pi\[LN^{-1}L^\top + \Si Q\Si + K(I_n+\Si R)^{-1}\Si K^\top \\
  &\hp{=-\,\Pi\[~}- K(I_n+\Si R)^{-1}\Si(I_n+R\Si)^{-1}K^\top\]\Pi^\top ds \\
  &\hp{=\ } - \Pi\[K(I_n+\Si R)^{-1}\Si + \Si(I_n+R\Si)^{-1}K^\top\]\Pi^\top dW.
\end{align*}
Integration from $t$ to $T$ and then taking conditional expectations with respect to $\cF_t$ on both sides, we obtain
\begin{align*}
\Pi(t)\Si(t)\Pi(t)^\top &= \dbE\bigg\{\int_t^T\Pi\[LN^{-1}L^\top + \Si Q\Si + K(I_n+\Si R)^{-1}\Si K^\top \\
&\hp{=\dbE\bigg\{\Pi\[~} - K(I_n+\Si R)^{-1}\Si(I_n+R\Si)^{-1}K^\top\]\Pi^\top ds\bigg|\cF_t\bigg\}.
\end{align*}
Observe that
\begin{align*}
& LN^{-1}L^\top + \Si Q\Si + K(I_n+\Si R)^{-1}\Si K^\top - K(I_n+\Si R)^{-1}\Si(I_n+R\Si)^{-1}K^\top \\
&\q= LN^{-1}L^\top + \Si Q\Si + K\[(I_n+\Si R)^{-1}\Si R\Si(I_n+R\Si)^{-1}\]K^\top \\
&\q= \h L \h L^\top,
\end{align*}
and that $\Pi(t,s)$ is independt of $\cF_t$ for $s\ges t$. Then we have
\begin{align*}
\Si(t) &= \Pi(t)^{-1}\lt\{\dbE\int_t^T\Pi(s)\h L(s)\h L(s)^\top\Pi(s)^\top ds\bigg|\cF_t\rt\}\[\Pi(t)^{-1}\]^\top \\
       &= \dbE\lt\{\int_t^T\Pi(t,s)\h L(s)\big[\Pi(t,s)\h L(s)\big]^\top ds\bigg|\cF_t\rt\} \\
       &= \dbE\lt\{\int_t^T\Pi(t,s)\h L(s)\big[\Pi(t,s)\h L(s)\big]^\top ds\rt\}.
\end{align*}
This completes the proof.
\end{proof}

{\it Proof of \autoref{thm:main}.}
First note that the state of system \rf{hatx} can also be transferred to $(T,\eta)$ from the
stochastic linear manifold $\cH(F,b)$ (\autoref{prop:x=hatx}) and that the controllability
Gramian of system \rf{hatx} over $[t,T]$ is $\Si(t)$ (\autoref{prop:Gramian-Si}).
Thus, by \autoref{coro:controllability} (iii), there exists $\xi\in L^2_{\cF_t}(\Om;\dbR^k)$ satisfying
\bel{5.24-1} F\Si(t)\xi = b-F\dbE[\Pi(t,T)\eta|\cF_t], \ee
where $\Pi(t,T)=\Pi(t)^{-1}\Pi(T)$ with $\Pi=\{\Pi(s);0\les s\les T\}$ being the solution of \rf{eqn:Pi}.
Applying \autoref{lemma:Y-formula} to the BSDE \rf{BSDE-f}, we obtain
$$\dbE[\Pi(t,T)\eta|\cF_t] = -\f(t),$$
and hence \rf{5.24-1} becomes
$$F\Si(t)\xi = F\f(t) + b.$$
Now using the identity
$$ I_n-\Si(t)[I_n+G\Si(t)]^{-1}G = [I_n+\Si(t)G]^{-1}, $$
it is straightforward to verify that
$$\l = [I_n+G\Si(t)]\xi - G\f(t)$$
is a solution of
$$\Big\{F\Si(t)[I_n+G\Si(t)]^{-1}\Big\}\l =  F[I_n+\Si(t)G]^{-1}\f(t)+b.$$
That is, $F[I_n+\Si(t)G]^{-1}\f(t)+b$ lies in the range of $F\Si(t)[I_n+G\Si(t)]^{-1}$. Since
$$F\Si(t)[I_n+G\Si(t)]^{-1} = F[I_n+\Si(t)G]^{-1}\Si(t)$$
and $F[I_n+\Si(t)G]^{-1}\Si(t)$ and $F[I_n+\Si(t)G]^{-1}\Si(t)F^\top$ have the same range (\autoref{lmm:range=range}),
we see that $F[I_n+\Si(t)G]^{-1}\f(t)+b$ also lies in the range of $F[I_n+\Si(t)G]^{-1}\Si(t)F^\top$,
which means the algebraic equation \rf{Eqn-lamda} has a solution.

\medskip

For the second assertion, let $\l^*\in L^2_{\cF_t}(\Om;\dbR^k)$ and $y^*$ be the solution to the SDE \rf{eqn:y*}.
By \autoref{thm:solution-BLQ-lamda}, the process
$$v^*(s) \deq N(s)^{-1}L(s)^\top y^*(s), \q s\in[t,T]$$
is the optimal control of Problem (BLQ)$_{\l^*}$.
Further, let $(x^*,z^*)$ be the adapted solution to the BSDE
$$\left\{\begin{aligned}
  dx^*(s) &= (Ax^* + Kz^* + Lv^*)ds + z^* dW(s), \q s\in[t,T],\\
   x^*(T) &= \eta.
\end{aligned}\right.$$
We see from the proof of \autoref{thm:solution-BLQ-lamda} that
$(x^*,z^*)$ and $y^*$ have the following relation (recalling \rf{Decoupling}):
$$ x^*(s) = -[\Si(s)y^*(s)+\f(s)], \q z^*(s) = [I_n+\Si(s)R(s)]^{-1}[\Si(s)K(s)^\top y^*(s)-\b(s)],$$
from which we obtain
\begin{align}\label{x-lamada}
x^*(t)  &= -\Si(t)y^*(t) - \f(t) \nn\\
        &= -\Si(t)[I_n+G\Si(t)]^{-1}[F^\top\l^*-G\f(t)] - \f(t) \nn\\
        &= -[I_n+\Si(t)G]^{-1}\Si(t)[F^\top\l^*-G\f(t)] - \f(t) \nn\\
        &= -[I_n+\Si(t)G]^{-1}\Si(t)F^\top\l^* + [I_n+\Si(t)G]^{-1}\Si(t)G\f(t) - \f(t) \nn\\
        &= -[I_n+\Si(t)G]^{-1}\Si(t)F^\top\l^* - [I_n+\Si(t)G]^{-1}\f(t).
\end{align}
According to \autoref{prop:relation}, the optimal control
$$v^*(s) =N(s)^{-1}L(s)^\top y^*(s), \q s\in[t,T]$$
of Problem (BLQ)$_{\l^*}$ is also optimal for\autoref{CBLQ} if and only if $x^*(t)\in\cH(F,b)$.
Using \rf{x-lamada}, we see the latter holds if and only if $\l^*$ is a solution of \rf{Eqn-lamda}.
$\hfill\qed$

\end{document}